\documentclass[dvipdfmx]{amsart}
\usepackage{amsmath,amsthm,amssymb,amscd}
\usepackage[arrow,matrix]{xy}
\usepackage[dvips]{graphicx}
\usepackage{enumerate}
\usepackage{listings}

\title{Generic 1-connectivity of flag domains\\ in Hermitian symmetric spaces}                                 %<-------------------
\author{Tatsuki Hayama}                 %<-------------------  %<-------------------

   %<-------------------

         %<-------------------

\address{%
Tatsuki Hayama\\               %<-------------------
School of Business Administration\\
Senshu University\\
Higashimita 2-1-1\\
Tama, Kawasaki, Kanagawa 214-8580\\
Japan\\            %<-------------------
hayama@isc.senshu-u.ac.jp              %<-------------------
}

\theoremstyle{plain}
\numberwithin{equation}{section}
\newtheorem{Theorem}{Theorem}[section]
\newtheorem{Proposition}[Theorem]{Proposition}
\newtheorem{Corollary}[Theorem]{Corollary}
\theoremstyle{definition}
\newtheorem{Definition}[Theorem]{Definition}
\newtheorem{Remark}[Theorem]{Remark}

\def\rank{\mathop{\mathrm{rank}}\nolimits}
\def\dim{\mathop{\mathrm{dim}}\nolimits}

\def\bd{\mathop{\mathrm{bd}}\nolimits}

\def\sgn{\mathop{\mathrm{sgn}}\nolimits}

\def\bd{\mathop{\mathrm{bd}}\nolimits}
\def\cl{\mathop{\mathrm{c\ell}}\nolimits}
\def\Span{\mathop{\mathrm{Span}}\nolimits}

\def\diag{\mathop{\mathrm{diag}}\nolimits}

\def\CC{\mathbb{C}}

\def\RR{\mathbb{R}}

\def\calF{\mathcal{F}}

\def\calO{\mathcal{O}}

\def\fraka{\mathfrak{a}}
\def\frakb{\mathfrak{b}}
\def\frakg{\mathfrak{g}}

\def\frakk{\mathfrak{k}}
\def\frakt{\mathfrak{t}}
\def\frakq{\mathfrak{q}}
\def\frakp{\mathfrak{p}}

\def\frakh{\mathfrak{h}}
\def\fraku{\mathfrak{u}}

\newcommand{\mr}[1]{{\mathrm{#1}}}

\begin{document}

% Insert title

\maketitle

\begin{abstract}
    A flag domain is an open real group orbit in a complex flag manifold.
    It has been shown that a flag domain is either pseudoconvex or pseudoconcave.
    Moreover, generically $1$-connected flag domains are pseudoconcave.
    In this study, for flag domains contained in irreducible Hermitian symmetric spaces of type \textit{AIII} or \textit{CI}, we determine which pseudoconcave flag domain is generically $1$-connected.
\end{abstract}
\section{Introduction}
Let $G$ be a connected complex semisimple Lie group, and let $G_0$ be a real form of $G$.
An open $G_0$-orbit in a $G$-flag manifold is called a flag domain.
For example, a Hermitian symmetric domain is a flag domain.
According to \cite{HongEtAl18,HayamaEtAl19}, a flag domain is either pseudoconvex or pseudoconcave.
A pseudoconvex flag domain, such as a Hermitian symmetric domain, possesses plenty of global functions.
In contrast, any global function on a pseudoconcave flag domain is constant.
In this study, we investigate pseudoconcave flag domains, focusing on {\it generic $1$-connectivity}.

Let $K_0$ be a maximally compact subgroup of $G_0$.
By the Matsuki duality \cite{Matsuki79}, $G_0$-orbits correspond to $K$-orbits with complexification $K$ of $K_0$.
Through this correspondence, a flag domain $D$ contains a compact submanifold called the base cycle of $D$.
The base cycle and its $G$-transformations play an important role in the study of  pseudoconcave flag domains.
Let us choose a base point $z$ in the base cycle.
The isotropy subgroup $Q_z$ at $z$ is a parabolic subgroup, and we have a unique open $Q_z$-orbit in the ambient flag manifold $G(z)$.
We say that $D$ is generically $1$-connected if the open $Q_z$-orbit intersects with the base cycle.
Huckleberry \cite{Huckleberry10a} showed that a flag domain is pseudoconcave if it is generically $1$-connected.
He also showed that $D$ is generically $1$-connected if $K$ is a simple Lie group.
For example, all $SL(n,\RR)$-flag domains are generically $1$-connected.
However, the following problem is still open: {\it are all pseudoconcave flag domains generically $1$-connected?}
In this study, we provide an answer for flag domains contained in irreducible compact Hermitian symmetric spaces of type \textit{AIII} or \textit{CI}.

All flag domains contained in the Hermitian symmetric space under consideration correspond to the signature, and our results indicate that generic $1$-connectivity depends on the numerical condition of the signature.
In the case of type \textit{CI}, where $G_0=Sp(2n,\RR)$, few flag domains are generically $1$-connected: the Hermitian symmetric space contains $(n+1)$ flag domains, of which $(n-1)$ are pseudoconcave, and at most one of them is generically $1$-connected.
In contrast, in the case of type \textit{AIII}, where $G_0=SU(p,q)$, more than one flag domain can be generically $1$-connected: almost half of the flag domains under consideration are generically $1$-connected if $2p<q$.
We prove these by using combinatorics of the Weyl groups and their action on the roots.
Moreover, we consider the generic $1$-connectivity of a certain type of flag domain fibered over the flag domains in the Hermitian symmetric spaces.

\section*{Acknowledgement}
The author would like to thank H. Ochiai and S. Kaji  for their helpful discussions. 
The author is also grateful to the referee for valuable comments on simplifying the proofs of our results. 
This work was partially supported by JSPS KAKENHI (Grant Number 20K03540).

\section{Cycle Connectivity of Flag Domains}
In this section, we review pseudoconcavity, cycle connectivity, and generic $1$-connectivity.
Subsequently, we present combinatorial conditions that are equivalent to generic $1$-connectivity.
\subsection{Pseudoconcavity}
Let $X$ be a connected complex manifold.
Andreotti \cite{Andreotti74} defined pseudoconcavity as follows:
\begin{Definition}\label{a_pcc}
   $X$ is pseudoconcave if we can find a relatively compact open subset $Y\subset X$ such that at every point $z\in \bd{(Y)}$, a holomorphic map $\rho$ on the unit disk $\mathbb{D}$ to $\cl{(Y)}$ satisfying $\rho(0)=z$ and $\bd{(\rho(\mathbb{D}))}\subset Y$ exists.
\end{Definition}
This definition is weaker than the definition of \textit{$q$-pseudoconcavity} in \cite{AndreottiEtAl62}, where a smooth exhaustion is required for the definition.
Similar to the finiteness theorem of \cite{AndreottiEtAl62} for higher cohomologies of $q$-pseudoconcave manifolds, we have a weaker version of the finiteness theorem:
\begin{Proposition}[\cite{Andreotti74}]\label{An}
  If $X$ is pseudoconcave, then any global function on $X$ is constant, and $\dim_{\CC}H^0(X,\calF)<\infty$ for any coherent sheaf $\calF$.
\end{Proposition}
To prove this finiteness theorem, the maximum principle works essentially.

\begin{Remark}
  Higher cohomologies of a pseudoconcave flag domain have a significant meaning in several aspects.
  In Hodge theory, Green et al. \cite{GreenEtAl13} studied them with specific Mumford-Tate domains in connection with automorphic cohomology.
  In representation theory, higher cohomologies give a geometric realization of Zuckerman derived functor modules, see \cite{Kobayashi98a} and references therein.
\end{Remark}

\subsection{Cycle connectivity}
Let $G$ be a connected complex Lie group.
For a $G$-flag manifold $Z$, we fix a base point $z\in Z$.
Then, $Z\cong G/Q_z\cong G(z)$, where $Q_z$ is the parabolic subgroup that stabilizes $z$.
Let $\frakg_0$ be a real form of the Lie algebra $\frakg$ of $G$, and let $\tau$ be the associated complex conjugation.
The $\tau$-invariant complex subspace $\frakq_z\cap \tau\frakq_z$ contains a $\tau$-stable Cartan subalgebra $\frakh$, where $\frakq_z$ is the Lie algebra of $Q_z$.
For a $\frakh$-root system $\Sigma$ of $\frakg$, we choose a positive root system $\Sigma^+$ such that $\frakq_z$ contains the Borel subalgebra $\frakb=\frakh\oplus\bigoplus_{\alpha\in\Sigma^+} \frakg_\alpha$.
Let $\Psi$ be the simple root system corresponding to $\Sigma^+$.
For a subset $\Phi\subset \Psi$, we define
\begin{align*}
  &\Phi_{\mr{r}}=\{
    \sum_{\psi\in\Psi} \epsilon_{\psi}\psi\in \Sigma \mid \epsilon_{\psi}= 0 \text{ whenever }\psi\not\in \Phi\}\\
  &\Phi^{\pm}_{\mr{n}}=\{\alpha\in \pm\Sigma^+\mid \alpha\not\in \Phi_{\mr{r}}\}.
\end{align*}
We may choose $\Phi$ such that  $\Sigma(\frakq_z)=\{\alpha\in \Sigma\; |\; \frakg_{\alpha}\subset \frakq_z \}=\Phi_{\rm{r}}\cup\Phi_{\rm{n}}^+$.
Here $\Phi_{\rm{r}}$ (resp. $\Phi^{+}_{\rm{n}}$) is reductive (resp. nilpotent) part of $\Sigma(\frakq_z)$.

Let $G_0$ be the real form of $G$ corresponding to $\frakg_0$.
By \cite{Aomoto66,Wolf69}, $G_0$-orbits in $Z$ are finitely many, and there is an open orbit.
An open $G_0$-orbit is called a flag domain.
Suppose that $D=G_0(z)$ is a flag domain.
Let $\theta$ be a Cartan involution that commutes with $\tau$.
Then, we may assume that $\frakh$ and $\Sigma^+$ satisfy the following conditions (see \cite[Theorem 4.5]{Wolf69}):
\begin{itemize}
  \item $\frakh_0=\frakh\cap\frakg_0$ is a $\theta$-stable maximally compact Cartan subalgebra of $\frakg_0$;
  \item $\tau \Sigma^+=-\Sigma^+$.
\end{itemize}
For the compact subgroup $K_0=G_0^{\theta}$ and its complexification $K$, $K_0$-orbit $K_0(z)$ coincides with the $K$-orbit, and $C_0=K_0(z)=K(z)$ is a complex compact manifold (see \cite[Theorem 4.3.1]{FelsEtAl06}).
Here $C_0$ is called the base cycle.

For any point $x,y\in D$, we write $x\sim y$ if there exists $C_i=g_i(C_0)\subset D$ with $g_i\in G$ such that $x\in C_1$ and $y\in C_N$, where the chain $C_1\cup\cdots \cup C_N$ is connected.
The relation $\sim$ is an equivalence relation, and $D/{\sim}$ is classified into two types:

\begin{Proposition}[\cite{Huckleberry10a}]\label{cycle}
  $D/{\sim}$ is either a Hermitian symmetric domain or a point.
  In the former case, $D$ is pseudoconvex.
  In the latter case, we say $D$ is {\it cycle connected}.
\end{Proposition}
Because a holomorphic function on $D$ is factored as $D\to D/{\sim}\to\CC$, the flag domain $D$ is cycle connected if and only if any global function on $D$ is constant.
Moreover, by Proposition \ref{An}, pseudoconcave flag domains are cycle connected.
Rather, the following theorem holds.
\begin{Theorem}[\cite{HongEtAl18,HayamaEtAl19}]
  A flag domain $D$ is cycle connected if, and only if, $D$ is pseudoconcave.
\end{Theorem}
\subsection{Generic $1$-connectivity}
Let $W=W(G,H)$ be the Weyl group with the Cartan subgroup $H=\exp{(\frakh)}$, and let $W_{\Phi}$ be the subgroup generated by the simple reflections associated with $\Phi$.
By the Bruhat decomposition, $Q_z$-orbits in $Z$ are parameterized by $W_{\Phi}\backslash W/W_{\Phi}$, and there is a unique open $Q_z$-orbit $\calO$.
\begin{Definition}
  A flag domain $D$ is generically $1$-connected if $C_0\cap\calO\neq \emptyset$.
\end{Definition}
The preimage of the base point under $D\to D/{\sim}$ contains an open subset if $D$ is generically $1$-connected.
Then, by Proposition \ref{cycle}, we have the following corollary:
\begin{Corollary}
  Generically $1$-connected flag domains are cycle connected (or equivalently pseudoconcave).
\end{Corollary}
The above corollary implies generic $1$-connectivity is a kind of cycle connectivity.
While cycle connectivity guarantees that any two points are connected by a chain of cycles of finite length, generic $1$-connectivity ensures that any point in $\calO$  is connected to the base point by a chain of length $1$.
In fact, if $D$ is generically $1$-connected, any point in $\calO$ is written as $g(z')$ with $z'\in C_0\cap \calO$ and $g\in Q_z$.
Then both $z$ and $g(z')$ are contained in $g(C_0)$.

Now $K$ is a reductive subgroup of $G$.
Because $C_0$ is a projective variety, $Q^K_z=K\cap Q_z$ is a parabolic subgroup.
Then $C_0$ can be decomposed into the disjoint union of $Q_z^K$-orbits.
Let $\frakh_0=\frakt_0\oplus \fraka_0$ be the $\theta$-stable decomposition, where $\frakt_0=\frakh_0\cap\frakk_0$ and $\fraka_0=\frakh_0\cap\frakp_0$ with the Cartan decomposition $\frakg_0=\frakk_0+\frakp_0$.
Because $\frakh_0$ is maximally compact, $\frakt_0$ is a maximal abelian subalgebra of $\frakk_0$.
Let $W_K=W(K_0,T_0)$ be the Weyl group with the maximal torus $T_0=\exp(\frakt_0)$.
Then, each $Q_z^K$-orbit in $C_0$ is the orbit at $w(z)$ with some $w\in W_K$.
Let $w_0^K$ be the longest element in $W_K$ with respect to the simple root system corresponding to the Borel subgroup contained in $Q_z^K$.
Then the $Q_z^K$-orbit at $w^K_0(z)$ is open.
\begin{Proposition}\label{W_K}
  The following conditions are equivalent:
  \begin{enumerate}
    \item $D$ is generically $1$-connected;
    \item $w^K_0(z)\in \calO$;
    \item $w^K_0(\Phi^-_{\mr{n}})\cap \Phi^-_{\mr{n}}=\emptyset$.
  \end{enumerate}
\end{Proposition}
\begin{proof}
  First, we show the equivalence between (1) and (2).
  If $D$ is not generically $1$-connected, $\calO\cap C_0=\emptyset$, and thus $w(z)\not\in \calO$ for all $w\in W_K$.
  Contrastingly, if $D$ is generically $1$-connected, $\calO\cap C_0$ is a $Q_z^K$-invariant subset that must contain the open $Q_z^K$-orbit in $C_0$.
  Hence $w^K_0(z)\in\calO$.

  Next, we show the equivalence between (2) and (3).
  We write $z'=w^K_0(z)$.
  Since $\Sigma=w^K_0(\Sigma(\frakq_z)\cup \Phi_{\mr{n}}^-)=\Sigma(\frakq_{z'})\cup w^K_0(\Phi_{\mr{n}}^-)$, we have 
  \begin{align*}
    \dim{(Q_z(z'))}=|\Sigma(\frakq_z)\cap w_0^K (\Phi_{\mr{n}}^-)|&=|w_0^K (\Phi_{\mr{n}}^-)|-|\Phi_{\mr{n}}^-\cap w_0^K (\Phi_{\mr{n}}^-)|\\&=\dim{D}-|\Phi_{\mr{n}}^-\cap w_0^K (\Phi_{\mr{n}}^-)|.
  \end{align*}
  Then $Q_z(z')$ is open if, and only if, $w^K_0(\Phi^-_{\mr{n}})\cap \Phi^-_{\mr{n}}=\emptyset$.
\end{proof}
For the longest element $w_0\in W$, the $Q_z$-orbit at $w_0(z)$ is open.
Moreover, any element $w\in W$ such that $w(z)$ contained in $\calO$ is written as $w=w_1 w_0 w_2$ with $w_1,w_2\in W_{\Phi}$.
If $\frakh_0$ is compact, that is, $\frakh_0=\frakt_0$, then $W_K$ is a subset of $W$.
In this case, $w_0^K$ is written as $w_0^K=w_1 w_0 w_2$ if and only if $w_0^K(z)\in \calO$.
By Proposition \ref{W_K}, we have the following corollary:
\begin{Corollary}\label{cor}
  In the case where $\frakh_0$ is compact, $D$ is generically $1$-connected if, and only if, there exists $w_1,w_2 \in W_{\Phi}$ such that $w_1 w_0^K w_2$ is the longest element in $W$.
\end{Corollary}
\section{Flag Domains in Hermitian Symmetric Spaces}
In this section, we suppose that $Z$ is an irreducible Hermitian symmetric space of compact type.
We then have the dual Hermitian symmetric domain $G_0/K_0$.
To state our result, we review the root structure of $\frakg$.
Let $\frakh_0$ be a maximal abelian subalgebra of $\frakk_0$.
We can choose a simple $\frakh$-root system $\{\psi_1,\ldots , \psi_n\}$ such that only one root is noncompact and compact otherwise. 
We suppose $\psi_m$ is noncompact.
Then, the set $\Sigma$ of roots can be decomposed into $\Sigma=\Sigma_{\mr{c}}\cup \Sigma_{\mr{nc}}^+\cup \Sigma_{\mr{nc}}^-$, where 
$$\Sigma_{\rm{c}}=\{\sum \epsilon_i\psi_i\mid \epsilon_m=0\},\quad \Sigma_{\rm{nc}}^{\pm}=\{\sum \epsilon_i\psi_i\mid \epsilon_m=\pm 1\}. $$
Let $\frakp^{\pm}=\sum_{\alpha\in \Sigma_{\rm{nc}}^{\pm}}\frakg_{\alpha}$ and $P^{\pm}=\exp{(\frakp^{\pm})}$.
Then, we have $Z\cong G/KP^+$, and the Hermitian symmetric domain $G_0/K_0$ is regarded as the $G_0$-orbit at the identity coset $z_0\in Z$.

In the Hermitian symmetric space $Z$, all $G_0$-orbits are related by the Cayley transforms.
Choosing a maximal set $\Xi=\{\xi_1,\ldots ,\xi_r \}\subset \Sigma_{\mr{nc}}^+$ of strongly orthogonal roots with $r=\rank_{\RR}{\frakg_0}$, the partial Cayley transform $c_{\xi}$ and the product $c_{\Gamma}=\prod_{\xi\in\Gamma} c_{\xi}$ is constructed from $\xi\in \Gamma\subset\Xi$ (see \cite{Wolf69,WolfEtAl97} for this construction).
For disjoint subsets $\Gamma, \Delta\subset \Xi$, we define $z_{\Gamma,\Delta}=c_{\Gamma}c_{\Delta}^2(z_0)$.
By \cite{Wolf69,WolfEtAl97}, the following properties hold:
\begin{itemize}
  \item Every $G_0$-orbit on $Z$ is written as $G_0(z_{\Gamma,\Delta})$ with some $\Gamma, \Delta\subset \Xi$;
  \item $G_0(z_{\Gamma,\Delta})=G_0(z_{\Gamma',\Delta'})$ if and only if $|\Gamma|=|\Gamma'|$ and $|\Delta|=|\Delta'|$;
  \item $G_0(z_{\Gamma,\Delta})$ is open if and only if $\Gamma=\emptyset$
\end{itemize}
Then, any flag domain in $Z$ is written as $G_0(z_{\emptyset,\Delta})$, and it depends on the cardinality of $\Delta$.

We choose $\Delta$ as the set $\{\xi_1,\ldots , \xi_a\}$ with $1\leq a\leq r$.
The square $c^2_{\Delta}$ of the partial Cayley transform is $s_{\Delta}=\prod_{1\leq i\leq a}s_{i}$ with the reflection $s_{i}$ with respect to $\xi_i\in \Delta$.
We set $z_a=s_{\Delta}(z_0)$ as a base point.
Then the $G_0$-orbit $D_a=G_0(z_a)$ is a flag domain in $Z$.
We denote by $\frakq_a$ the parabolic subalgebra at $z_a$.
Here, $\frakq_a=s_{\Delta}(\frakk +\frakp_+)$, which contains the Borel subalgebra corresponding to the simple root system $\Psi=\{s_{\Delta}(\psi_1),\ldots ,s_{\Delta}(\psi_n)\}$.
The set $\Sigma(\frakq_a)$ of roots is decomposed as 
$$\Sigma(\frakq_a)=\Phi_{\rm{r}}\cup\Phi_{\rm{n}}^+\quad \text{with }\Phi=\Psi\setminus \{s_{\Delta}(\psi_m)\}.$$
Now $\Phi_{\rm{n}}^-$ is decomposed as $\Phi_{\rm{n}}^-=(\Phi_{\rm{n}}^-\cap\Sigma_{\rm{c}})\cup (\Phi_{\rm{n}}^-\cap\Sigma_{\rm{nc}})$.
Using the longest element $w_K^0$, we have $w_K^0(\Phi_{\rm{n}}^-\cap\Sigma_{\rm{c}})\subset \Sigma_{\rm{c}}$ and $w_K^0(\Phi_{\rm{n}}^-\cap\Sigma_{\rm{c}})\cap (\Phi_{\rm{n}}^-\cap\Sigma_{\rm{c}})=\emptyset$.
Then, to show generic $1$-connectivity, Proposition \ref{W_K} is simplified as follows:
\begin{Corollary}\label{herm}
  $D_a$ is generically $1$-connected if and only if $w^K_0(\Phi_{\rm{n}}^-\cap\Sigma_{\rm{nc}})\cap (\Phi^-_{\rm{n}}\cap\Sigma_{\rm{nc}}) =\emptyset$.
\end{Corollary}

The Hermitian symmetric space with a classical group $G$ can be classified into four types (see \cite[Chapter VII]{Helgason01a} for details): \textit{AIII}, \textit{DIII}, \textit{BDI}, and \textit{CI}.
We consider the $1$-connectivity of $D_a$ in the cases of type \textit{AIII} and \textit{CI}.

\subsection{Case for type \textit{CI}}
We fix a symplectic form $\omega$ on $\RR^{2n}$, and let $G_0=Sp(2n,\RR)$ be the subgroup of $SL(2n,\RR)$, leaving invariant this form.
We have a basis $\{f_i\}_i$ that satisfies $\sqrt{-1}\omega(v ,\bar{w})=\sum_{i\leq n}(v_iw_{i}-v_{n+i}w_{i+i})$ for $v=\sum v_if_i$ and $ w=\sum w_if_i$.
Using this basis, we may regard $G_0$ as $U(n,n)\cap G$.

The Grassmannian $Z$ of $\omega$-isotropic $n$-planes in $\CC^{2n}$ is a $G$-flag manifold, and all open $G_0$-orbits correspond to pairs of numbers of positive and negative signatures of the associated Hermitian form.
Let 
$$z_a=\Span{\{f_i\mid a<i\leq n+a\}}\in Z$$
for $0\leq a\leq n$, where $\sgn{(z_a)}=(n-a,a)$.
Then, any flag domain is written as $D_a=G_0(z_a)$, and both $D_0$ and $D_n$ are Hermitian symmetric domains, that is, each one is the Siegel upper (or lower) half space.
\begin{Theorem}\label{CI}
  An $Sp(2n,\RR)$-flag domain $D_{a}$ in the Hermitian symmetric space is generically $1$-connected if and only if $2a=n$.
\end{Theorem}
For this proof, we consider the root structure and the Weyl group action.
We choose $\frakk_0=\fraku(2n)\cap \frakg_0\cong \fraku(n)$, and let $\frakh_0\subset \fraku(n)$ be the maximal torus consisting of diagonal matrices.
We define $e_i\in\frakh^*$ by $e_i(X)=a_i$ for $X=\diag{(a_1,\ldots ,a_{n})}\in \frakh$.
Then, we may choose the simple root system $\{\psi_1,\ldots ,\psi_{n}\}$, where $\psi_i=e_i-e_{i+1}$ for $i<n$ and $\psi_n=2e_n$, and we can write
\begin{align*}
  &\Sigma_{\rm{c}}=\{\sum \epsilon_i\psi_i\mid \epsilon_n=0\}=\{e_i-e_j\mid i\neq j\},\\
  &\Sigma_{\rm{nc}}^{\pm}=\{\sum \epsilon_i\psi_i\mid \epsilon_n=\pm 1\}=\{\pm(e_i+e_j)\mid 1\leq i\leq j\leq n\}
\end{align*}

Now $G_0$ is the split real form, i.e. $\rank_{\RR}{\frakg_0}=\dim{\frakh_0}=n$.
The maximal set of noncompact orthogonal roots is $\{\xi_1,\ldots , \xi_n\}$ with $\xi_i=2e_i$.
Then $s_{\Delta}(e_i)=-e_i$ if $i<a$, and $s_{\Delta}(e_i)=e_i$ otherwise.
Since $-s_{\Delta}(\psi_i)\in \Sigma_{\rm{c}}\cup\Sigma_{\rm{nc}}^+$, i.e. $-\psi_i\in s_{\Delta}(\Sigma_{\rm{c}}\cup\Sigma_{\rm{nc}}^+)$, for $1\leq i\leq n-1$, 
$\Sigma(\frakq_a)=s_{\Delta}(\Sigma_{\rm{c}}\cup\Sigma_{\rm{nc}}^+)$ contains the set $\{-\psi_i\}_{i\neq n}$ of simple roots of $\Sigma_{\rm{c}}$.
By using these simple roots, $W_K$ is the symmetry group $S_n$ in terms of the permutations of the indices of $e_1,\ldots ,e_n$, and the longest element in $W_K$ is
\begin{align}\label{wK0}
  w^K_0=\bigl(
  \begin{smallmatrix}
    1 & 2 & \cdots & n-1 & n \\
    n & n-1 & \cdots &  2  & 1
  \end{smallmatrix}
  \bigr).
\end{align}

\begin{proof}[Proof of Theorem \ref{CI}]
We apply Corollary \ref{herm}.  
Since $\Phi_{\rm{n}}^-=s_{\Delta}(\Sigma_{\rm{nc}}^-)$, we have
$$\Phi_{\rm{n}}^-\cap\Sigma_{\rm{nc}}=\{e_i+e_j\mid 1\leq i\leq j\leq a\}\cup\{-e_i-e_j\mid a+1\leq i\leq j\leq n\}.$$
Then, $w^K_0(\Phi_{\rm{n}}^-\cap\Sigma_{\rm{nc}})\cap (\Phi^-_{\rm{n}}\cap\Sigma_{\rm{nc}}) =\emptyset$ if and only if $2a=n$.
\end{proof}

Next, we consider the generic $1$-connectivity of a certain type of flag domain fibered over $D_a$ with $0< a<n$.
Let $Z_{m}$ be the flag manifold consisting of sequences $V_1\subset V_2$ of $\omega$-isotropic subspaces with $0<\dim{V_1}=m\leq a< \dim{V_2}=n$.
We set
\begin{align}\label{F}
  z_{m,0}=(F_m\subset F_n)\in Z_m \quad \text{where }F_r=\Span{\{f_i\}_{i\leq r}}.
\end{align}
We continue to assume $\Delta=\{\xi_1,\ldots ,\xi_a\}$ and set $z_{m,a}=s_{\Delta}(z_{m,0})$ as the base point.
Then, $D_{m,a}=G_0(z_{m,a})$ is the flag domain in $Z_m$, determined by 
$$\sgn{(V_1)}=(0,m),\quad \sgn{(V_2)}=(n-a,a).$$

\begin{Proposition}\label{CI2}
  The flag domain $D_{m,a}$ is not generically $1$-connected.
\end{Proposition}
\begin{proof}
  We denote by $\frakq_{m,a}$ the parabolic subalgebra at $z_{m,a}$.
  Then, $\frakq_{m,a}$ contains the Borel subalgebra corresponding to the simple root system $\Psi=\{s_{\Delta}(\psi_1),\ldots ,s_{\Delta}(\psi_n)\}$, where  $\Sigma(\frakq_{m,a})=\Phi_{\rm{r}}\cup\Phi_{\rm{n}}^+$ with $\Phi=\Psi\setminus \{s_{\Delta}(\psi_m),s_{\Delta}(\psi_n)\}$.
  We have 
  $$e_1+e_n=s_{\Delta}(-e_1+e_n)=-\sum_{i=1}^{n-1}s_{\Delta}(\psi_i)\in \Phi_{\rm{n}}^-.$$
  Moreover, $\Sigma(\frakq_{m,a})$ contains the set $\{-\psi_i\}_{i\neq n}$ of simple roots of $\Sigma_{\rm{c}}$, and the longest element $w_0^K$ is defined in (\ref{wK0}).
  Then, $ w_0^K(e_1+e_n)=e_1+e_n$; hence, $D_{m,a}$ is not generically $1$-connected by Proposition \ref{W_K}.
\end{proof}

\subsection{Case for type \textit{AIII}}
We fix a Hermitian form $\langle \bullet,\bullet\rangle$ on $\CC^{p+q}$, and let $G_0=SU(p,q)$ be the subgroup of $SL(p+q,\CC)$, leaving invariant this form.
We may assume $p\leq q$.
We have a basis $\{f_i\}_i$ that satisfies $\langle v,w\rangle=\sum_{i\leq p}v_iw_i-\sum_{j> p}v_jw_j$ for $v=\sum v_if_i$ and $w=\sum w_if_i$.

The Grassmannian $Z$ of the $p$-planes in $\CC^{p+q}$ is a $G$-flag manifold, and all open $G_0$-orbits correspond to pairs of numbers of positive and negative signatures.
Let 
$$z_a=\Span{\{f_i\mid a<i\leq p, \; p+q-a<i\leq p+q\}}$$
for $0\leq a\leq p$, where $\sgn{(z_a)}=(p-a,a)$.
Then, any flag domain is written as $D_a=G_0(z_a)$, and $D_0$ is the Hermitian symmetric domain $\{X\in M_{p,q}(\CC)|\; I-{}^t\bar{X}X> 0\}$.
\begin{Theorem}\label{AIII}
  An $SU(p,q)$-flag domain $D_a$ in the Hermitian symmetric space is generically $1$-connected if and only if $p\leq 2a\leq q$.
\end{Theorem}
As in type \textit{CI}, we consider the root structure and Weyl group action.
We choose
$\frakk_0=\mathfrak{su}(p,q)\cap\fraku(p+q)
$,
and let $\frakh_0\subset\frakk_0$ be the maximal torus consisting of diagonal matrices.
Then, we may choose the simple root system $\{\psi_1,\ldots ,\psi_{p+q-1}\}$, where $\psi_i=e_i-e_{i+1}$.
The simple root $\psi_i$ is compact if $i=p$ and is  noncompact otherwise.
Then we can write
\begin{align*}
  &\Sigma_{\rm{c}}=\{\sum \epsilon_i\psi_i\mid\epsilon_p=0\}=\{e_i-e_j\mid  i,j\leq p\text{ or }p< i,j\},\\
  &\Sigma^\pm_{\rm{nc}}=\{\sum \epsilon_i\psi_i\mid\epsilon_p=\pm 1\}=\{\pm(e_i-e_j)\mid   i\leq p<j\}.
\end{align*}

We set the maximal set 
$\{\xi_1,\ldots ,\xi_p\}$
of strongly orthogonal noncompact roots with
$\xi_i=e_i-e_{p+q+1-i}$.
The reflection $s_i$ with respect to $\xi_i$ is the permutation of the indices of $e_1,\ldots ,e_{p+q}$, which exchanges $i$ and $p+q+1-i$.
Then $s_{\Delta}=\prod_{i=1}^a s_i$ is 
the permutation 
\begin{equation}
s_{\Delta}=
\left(
\begin{smallmatrix}
  1&\cdots &a&a+1&\cdots &p+q-a&p+q-a+1&\cdots &p+q\\
  p+q&\cdots &p+q-a+1&a+1&\cdots &p+q-a&a&\cdots &1
\end{smallmatrix}
\right).\label{s_D}
\end{equation}
Therefore $\Sigma(\frakq_a)=s_{\Delta}(\Sigma_{\rm{c}}\cup\Sigma_{\rm{nc}}^+)$ contains the set $\{-\psi_i\}_{i\neq p}$ of simple roots of $\Sigma_{\rm{c}}$.
By using these simple roots, $W\cong S_{p+q}$ and $W_K\cong S_p\times S_q$.
The longest element in $W_K$ is the permutation
\begin{align}\label{wK0_}
  w_0^K=\left(
\begin{smallmatrix}
  1&\cdots&p&p+1&\cdots &p+q\\
  p&\cdots&1&p+q&\cdots &p+1
\end{smallmatrix}\right).
\end{align}
\begin{proof}[Proof of Theorem \ref{AIII}]
Similar to the proof of Theorem \ref{CI}, we have
\begin{align}\label{su_root}
  \Phi^-_{\rm{n}}\cap\Sigma_{\rm{nc}}=&
 \{e_i-e_j\mid 1\leq i\leq a,\; p+q-a+1 \leq j\leq p+q \}\\ \nonumber
 &\cup \{-e_i+e_j\mid a+1\leq i\leq p,\; p+1\leq j\leq p+q-a\}.
\end{align}
Then, it is immediate to verify that $w^K_0(\Phi^-_{\rm{n}}\cap\Sigma_{\rm{nc}})\cap (\Phi^-_{\rm{n}}\cap\Sigma_{\rm{nc}})=\emptyset $ if and only if $p\leq 2a\leq q$.
Hence Theorem \ref{AIII} follows from Corollary \ref{herm}.
\end{proof}

Next, we consider the generic $1$-connectivity of a certain type of flag domain fibered over $D_a$ with $p\leq 2a\leq q$.
We define sequences $\mathbf{s}:s_1\leq\cdots\leq s_{a+1}$ and $\mathbf{t}:t_1\geq\cdots\geq t_{a+1}$ with
$$
s_i=\begin{cases}
  i &\text{if }i\leq a\\
  p &\text{if }i=a+1,
\end{cases}
\quad 
t_i=\begin{cases}
  p+q-i &\text{if }i\leq a\\
  p &\text{if }i=a+1.
\end{cases}
$$
Let $Z_\mathbf{s}$ (resp. $Z_\mathbf{t}$) be a flag manifold consisting of sequences $V_1\subset \cdots \subset V_{a+1}$ with $\dim{V_i}=s_i$ (resp. $V_1\supset \cdots \supset V_{a+1}$ with $\dim{V_i}=t_i$).
We let
\begin{align*}
  z_{\mathbf{s},0}=(F_1\subset \cdots \subset F_a\subset F_p),\quad 
  z_{\mathbf{t},0}=(F_{p+q-1}\supset \cdots F_{p+q-a}\supset F_p)
\end{align*}
where $F_i$ is defined as in (\ref{F}).
We continue to assume $\Delta=\{\xi_1,\ldots ,\xi_a\}$ and set $z_{\mathbf{s}}=s_{\Delta}(z_{\mathbf{s},0})$ and $z_{\mathbf{t}}=s_{\Delta}(z_{\mathbf{t},0})$. 
Then, $D_{\mathbf{s}}=G_0(z_{\mathbf{s}})$ (resp. $D_{\mathbf{t}}=G_0(z_{\mathbf{t}})$) is the flag domain in $Z_\mathbf{s}$ (resp. $Z_\mathbf{t}$) determined by 
$$\sgn{(V_i)}=\begin{cases}
  (0,i)&\text{if }i\leq a\\
  (p-a,a)&\text{if }i=a+1.\\
\end{cases}
\quad 
\text{\Bigg(resp. } 
\sgn{(V_i)}=\begin{cases}
  (p-i,q)&\text{if }i\leq a\\
  (p-a,a)&\text{if }i=a+1.\\
\end{cases}
\text{\Bigg)}$$
\begin{Proposition}
  The flag domains $D_{\mathbf{s}}$ and $D_{\mathbf{t}}$ are generically $1$-connected.
\end{Proposition}
\begin{proof}
  We denote by $\frakq_{\mathbf{s}}$ the parabolic subalgebra at $z_{\mathbf{s}}$.
  Then, $\frakq_{\mathbf{s}}$ contains the Borel subalgebra corresponding to the simple root system $\Psi=\{s_{\Delta}(\psi_i)\}_{1\leq i\leq p+q-1}$, and $\Sigma(\frakq_{\mathbf{s}})=\Phi_{\rm{r}}\cup\Phi_{\rm{n}}^+$ with 
  $\Phi=\Psi\setminus\{s_{\Delta}(\psi_1),\ldots ,s_{\Delta}(\psi_a),s_{\Delta}(\psi_p) \}$.
  By the permutation (\ref{s_D}), we have
  \begin{align*}
    \Phi_{\rm{n}}^-\cap \Sigma_{\rm{nc}}=&\text{(the right hand side of (\ref{su_root}))}\\
    &\cup \{e_i-e_j \mid  a+1\leq i\leq p,\;p+q-a+1\leq j\leq p+q\}.
  \end{align*}
  Moreover, $\frakq_{\mathbf{s}}$ contains the set $\{-\psi_i\}_{i\neq p}$ of simple roots of $\Sigma_{\rm{c}}$, and the longest element is the permutation (\ref{wK0_}).
  Then, we have $w^K_0(\Phi^-_{\rm{n}}\cap\Sigma_{\rm{nc}})\cap (\Phi^-_{\rm{n}}\cap\Sigma_{\rm{nc}})=\emptyset $ and $D_{\mathbf{s}}$ is generically 1-connected.

  For the proof of $D_{\mathbf{t}}$, the set $\Phi$ is replaced by $\Psi\setminus \{s_{\Delta}(\psi_{p+q-1}),\ldots ,s_{\Delta}(\psi_{p+q-a}),s_{\Delta}(\psi_p) \}$.
  Then, $\Phi_{\rm{n}}^-\cap \Sigma_{\rm{nc}}$ is written as
  $$\text{(the right hand side of (\ref{su_root}))}\cup \{e_i-e_j \mid  1\leq i\leq a,\;p+1\leq j\leq p+q-a\}.$$
  Therefore, $D_{\mathbf{t}}$ is generically 1-connected similarly as in the case for $D_{\mathbf{s}}$.
\end{proof}
Let $\mathbf{s'}$ (resp. $\mathbf{t'}$) be a subsequence of $\mathbf{s}$ (resp. $\mathbf{t}$), and we define the subsequences $z_{\mathbf{s'}}$ (resp. $z_{\mathbf{t'}}$) of $z_{\mathbf{s}}$ (resp. $z_{\mathbf{t}}$), as described above.
Then, we have the flag domain $D_{\mathbf{s'}}=G_0(z_{\mathbf{s'}})$ and $D_{\mathbf{t'}}=G_0(z_{\mathbf{t'}})$, which compose the fibration 
$$D_{\mathbf{s}}\rightarrow D_{\mathbf{s'}}\rightarrow D_a \leftarrow D_{\mathbf{t'}} \leftarrow D_{\mathbf{t}}$$ 
Because the parabolic subalgebra at $z_{\mathbf{s'}}$ (resp. $z_{\mathbf{t'}}$) contains $\frakq_{\mathbf{s}}$ (resp. $\frakq_{\mathbf{t}}$), the above proposition and Proposition \ref{W_K} imply the following corollary:
\begin{Corollary}
  The flag domains $D_{\mathbf{s'}}$ and $D_{\mathbf{t'}}$ are generically $1$-connected.
\end{Corollary}

% Your bilbigraphy           %<-------------------

\end{document}